\documentclass[11pt]{amsart}

\usepackage[dvipsnames]{xcolor}
\usepackage{amssymb,verbatim}
\usepackage{amsmath,amsfonts,enumitem,bm}
\usepackage[mathscr]{euscript} 
\usepackage{amsthm}
\usepackage{url}
\usepackage{graphicx} 
\usepackage{enumitem} 
\usepackage{hyperref} 
\hypersetup{colorlinks} 
\usepackage{bbm} 
\usepackage{bm} 
\usepackage{mathrsfs} 
\usepackage{cancel} 
\usepackage{ytableau} 
\usepackage{mathtools}
\usepackage{faktor} 
\usepackage{algorithm}
\usepackage{algorithmic}

\usepackage[colorinlistoftodos]{todonotes}

\RequirePackage{cleveref}
\usepackage{hypcap}
\hypersetup{colorlinks=true, citecolor=darkblue, linkcolor=darkblue}
\definecolor{darkblue}{rgb}{0.0,0,0.7}

\definecolor{darkred}{rgb}{0.68,0,0}

\definecolor{darkgreen}{rgb}{0,.38,0}

\newcommand{\defn}[1]{\emph{#1}}

\newcommand{\defng}[1]{\emph{#1}}

\setlist[enumerate]{
	label=\textnormal{({\roman*})},
	ref={\roman*}}

\makeatletter
\def\th@plain{%
	\thm@notefont{}
	\itshape 
}
\def\th@definition{%
	\thm@notefont{}
	\normalfont 
}
\makeatother

\newtheorem{thm}{Theorem}[section]
\newtheorem{lemma}[thm]{Lemma}

\newtheorem*{claim*}{Claim}
\newtheorem{cor}[thm]{Corollary}

\newtheorem{prop}[thm]{Proposition}

\newtheorem{conv}[thm]{Convention}

\theoremstyle{definition}

\numberwithin{figure}{section}
\numberwithin{equation}{section}

\def\zz{\mathbb Z}
\def\nn{\mathbb N}
\def\cc{\mathbb C}

\def\qqq{\mathbb Q}

\def\pp{\mathbb P}

\def\ov{\oa}

\def\la{\lambda}
\def\ga{\gamma}

\def\al{\alpha}
\def\be{\beta}
\def\om{\omega}

\def\ve{\varepsilon}

\def\vk{\vt}

\def\cW{\mathcal W}

\def\<{\langle}
\def\>{\rangle}

\def\SO{ {\text {\sf SO} } }

\def\SL{ {\text {\sf SL} } }

\def\oa{\overrightarrow}

\def\vt{\vartheta}

\def\rK{{\mathrm{K}}}

\def\0{{\mathbf 0}}

\def\ol#1{{\overline {#1}}}

\def\.{\hskip.06cm}
\def\ts{\hskip.03cm}

\def\bx{{\textbf{\textit{x}}}}

\def\bal{{\boldsymbol{\alpha}}}

\def\ze{{\zeta}}

\def\.{\hskip.06cm}
\def\ts{\hskip.03cm}

\def\nin{\noindent}

\newcommand{\textsu}[1]{\textup{\textsf{#1}}}
\usepackage{indentfirst}
\DeclareTextSymbolDefault{\ae}{T1}

\newcommand{\ComCla}[1]{\textup{\textsu{#1}}}

\newcommand{\sharpP}{\ComCla{\#P}}
\newcommand{\SP}{\ComCla{\#P}}

\newcommand{\GapP}{\ComCla{GapP}}

\newcommand{\Sigmap}{\ensuremath{\Sigma^{{\textup{p}}}}}
\newcommand{\Pip}{\ensuremath{\Pi^{{\textup{p}}}}}

\newcommand{\NP}{\ComCla{NP}}
\newcommand{\coNP}{\ComCla{coNP}}

\newcommand{\BPP}{\ComCla{BPP}}

\newcommand{\RP}{\ComCla{RP}}
\newcommand{\coRP}{\ComCla{coRP}}

\newcommand{\E}{\ComCla{E}}
\renewcommand{\P}{\ComCla{P}}
\newcommand{\CeqP}{\ComCla{C$_=$P}}

\newcommand{\PH}{\ComCla{PH}}

\newcommand{\PSPACE}{\ComCla{PSPACE}}

\newcommand{\AM}{\ComCla{AM}}
\newcommand{\coAM}{\ComCla{coAM}}

\def\SP{\sharpP}

\def\GRH{\textup{\sc GRH}}

\def\PIT{\textup{\sc PIT}}

\def\SV{\textup{\sc SchubertVanishing}}

\def\LRV{\textup{\sc LRVanishing}}

\def\poly{{\P}}
\def\CEP{{\CeqP}}

\newcommand{\inv}{\operatorname{{\ell}}}

\newcommand{\numInt}[1]{U({#1})}
\newcommand{\numGp}[1]{N({#1})}

\begin{document}

\title[Vanishing of Schubert coefficients]{Vanishing of Schubert coefficients \\
in probabilistic polynomial time}


 \author[Igor Pak \. \and \. Colleen Robichaux]{Igor Pak$^\star$  \. \and \.  Colleen Robichaux$^\dagger$}

\makeatletter

\thanks{\today}

\thanks{\thinspace ${\hspace{-.45ex}}^\star$Department of Mathematics,
UCLA, Los Angeles, CA 90095, USA. Email:  \texttt{pak@math.ucla.edu}}
\thanks{\thinspace ${\hspace{-.45ex}}^\dagger$Department of Mathematics,
UC Davis, Davis, CA 95616, USA. Email:  \texttt{robichaux@ucdavis.edu}}

\begin{abstract}
The \emph{Schubert vanishing problem} asks whether Schubert structure
constants are zero.  We give a complete solution of the problem from
an algorithmic point of view, by showing that Schubert vanishing can
be decided in probabilistic polynomial time.
\end{abstract}

\maketitle

\section{Introduction}\label{s:intro}

\subsection{Vanishing of Schubert coefficients} \label{ss:intro-main}
Determining Schubert structure constants (\emph{Schubert coefficients}) is one
of the oldest and most celebrated problems in enumerative geometry,
going back to Schubert's original work in 1870s, see \cite{Sch79}.
Motivated in part by \emph{Hilbert's 15th Problem} aiming to make Schubert's work rigorous
(see \cite{Kle76}), the area of Schubert calculus has exploded and developed
rich connections with representation theory and algebraic combinatorics
(see e.g.\ \cite{AF,BGP25,Knu22}).

In this paper we study the \emph{Schubert vanishing problem} \ts which asks
whether Schubert coefficients are zero.  This problem has remained
a major challenge for decades and remained unresolved despite significant study (see $\S$\ref{ss:intro-geo-back}, $\S$\ref{ss:intro-CS} and~$\S$\ref{ss:finrem-hist}).
We resolve it algorithmically, by
showing that deciding Schubert vanishing can be done in probabilistic
polynomial time.  This is an ultimate result of a series of our
previous papers \cite{PR-STOC,PR-CR,PR-Sigma}.

\smallskip

We start with a general setup, see e.g.~\cite{BH95,AF} for the background.
Let \ts ${\sf G}$ \ts be a {connected
semisimple complex Lie group}.
Take \ts ${\sf B}\subset {\sf G}$ and \ts ${\sf B}_{-}\subset {\sf G}$ \ts
to be the {Borel subgroup}
and {opposite Borel subgroup}, respectively.
The \defn{torus subgroup} \ts is defined as \ts ${\sf T}={\sf B}\cap {\sf B}_{-}\ts$.
The \defn{Weyl group} \ts is defined as
the normalizer \ts $\mathcal{W}\cong N_{\sf G}({\sf T})/{\sf T}$.
The \defng{Bruhat decomposition} \ts states that
$$
{\sf G} \ = \ \bigsqcup_{w\in \mathcal{W}} \. {\sf B}_{-} \. \dot{w} \. {\sf B}\ts,
$$
where \ts $\dot{w}$ \ts is the preimage of \ts $w$ \ts in the normalizer
\ts $N_{\sf G}(\sf T)$.

The \defn{generalized flag variety} \ts is defined as \ts  ${\sf G}/{\sf B}$.
Recall that ${\sf G}/{\sf B}$ \ts  has finitely many orbits under the left action of ${\sf B}_{-}\ts$.
These are called \defn{Schubert cells} \ts and denoted \ts $\Omega_w\ts$.
Schubert cells are indexed by the Weyl group elements \ts $w\in\mathcal{W}$.

 The \defn{Schubert varieties} $X_w$ are the Zariski closures of Schubert cells~$\Omega_w\ts$.
The \defn{Schubert classes} $\{\sigma_w\}_{w\in \mathcal{W}}$ are the
Poincar\'e duals of Schubert varieties. These form a ${\mathbb Z}$--linear basis
of the cohomology ring \ts $H^{*}({\sf G}/{\sf B})$.
The \defn{Schubert coefficients} \ts $c_{u,v}^w$ \ts are defined as structure constants:
\begin{equation}\label{eq:SchubStructureCoh}
    \sigma_u \smallsmile \sigma_v \, = \, \sum_{w\in {\mathcal{W}}} \. c_{u,v}^{w} \. \sigma_{w}\..
\end{equation}
Thus $c_{u,v}^{w}=[\sigma_{{w_\circ}}] \ts \sigma_u \smallsmile \sigma_v\smallsmile \sigma_{w_{\circ}w}$, where
\ts $w_\circ$ \ts is the \emph{long word} \ts in~$\ts\mathcal{W}$.
These are a special case of
\begin{equation}\label{eq:SchubStructureCoh-k}
    c(u_1,u_2,\ldots,u_k) \, := \, [\sigma_{{w_\circ}}]  \. \sigma_{u_1} \smallsmile\sigma_{u_2}\smallsmile\cdots  \smallsmile \sigma_{u_k}\.,
\end{equation}
where \ts $k\geq 3$.
In particular, we have \ts $c_{u,v}^w = c(u,v,w_\circ w)$. By commutativity of $H^{*}({\sf G}/{\sf B})$,
Schubert coefficients \ts $c(u_1,\ldots,u_k)$ \ts exhibit $S_k$-symmetry.

By \defng{Kleiman transversality} \cite{Kleiman}, the coefficients \ts $c(u_1,\ldots,u_k)$ \ts
count the number of points in the intersection of generically translated Schubert varieties whose total codimension is $\dim({\sf G}/{\sf B})$:
\begin{equation}\label{eq:SchubStructureInt}
    c(u_1,\ldots,u_k) \,  = \,\#\big\{\. X_{u_1}\big(F_{\bullet}^{(1)}\big) \cap \cdots \cap X_{u_k}\big(F^{(k)}_{\bullet}\big)\.\big\},
\end{equation}
where \ts $F_{\bullet}^{(i)}$ \ts are generic flags.
In particular, we have \ts $c(u_1,\ldots,u_k)\in \nn$.
The \defn{Schubert vanishing problem} \ts is the decision problem
$$
\SV \, := \, \big\{\ts c(u_1,\ldots,u_k)  =^? 0 \ts \big\},
$$
where $u_1,\ldots,u_k\in \cW$.  We consider the problem only for classical types
$Y \in \{A,B,C,D\}$, and use notation \ts $\SV(Y)$ \ts to denote the
Schubert vanishing problem in type~$Y$.\footnote{For non-classical types \ts $E_6$, \ts $E_7$,
\ts $E_8$, \ts $F_4$ \ts and \ts $G_2\ts$,  there is only a finite number of
Schubert coefficients, so the problem is uninteresting from the computational
complexity point of view.} These correspond to considering groups ${\sf G}\in\{{\sf SL}_n(\cc), {\sf SO}_{2n+1}(\cc), {\sf Sp}_{2n}(\cc), {\sf SO}_{2n}(\cc)\}$, respectively. We will use \ts $Y_n \in \{A_{n-1},B_n,C_n,D_n\}$,
where \ts $n=n({\sf G})$ \ts is the \emph{rank} \ts of Lie group~${\sf G}$.
Recall that \ts $c(u_1,\ldots,u_k)=0$ \ts for \ts $k> \ell(w_\circ)$, assuming each $u_i\neq \rm{id}$. Thus we
consider only the case when \ts $k\le \ell(w_\circ)$, where $\ell$ denotes
the \emph{length function} in~$\cW$.


\smallskip

\begin{thm}[{\rm Main theorem}{}]\label{t:main}
For each \ts $Y \in \{A,B,C,D\}$, the problem \ts $\SV(Y)$ \ts
can be decided in probabilistic polynomial time.  More precisely, for every fixed
\ts $\delta>0$ \ts and every \ts $0<\ve\le\frac13$,
there is a probabilistic algorithm which inputs elements \ts $u_1,\ldots,u_k \in \mathcal{W}(Y_n)$ \ts
and after \ts $O\big(n^{4.5+\delta}\log\frac{1}{\ve}\big)$ \ts
arithmetic operations outputs either $0$ or~$1$, such that:

$\circ$ \ if \. $c(u_1,\ldots,u_k) = 0$, then \. $\pp[\text{\em output} = 0] =1$, \. and

$\circ$ \  if \. $c(u_1,\ldots,u_k)>0$, then \. $\pp[\text{\em output} = 1]>1-\ve$.
\end{thm}
\smallskip

The proof is based on Purbhoo's criterion for Schubert vanishing
\cite[Cor.~2.6]{Purbhoo06}.  We show that the criterion is equivalent
to the degeneracy of certain determinant with polynomial entries.
It can then be tested in polynomial time whether this determinant
is identically zero by a random substitution of variables.  This explains
the one-sided error in the algorithm, since finding a nonzero evaluation
of the determinant \emph{guarantees} \ts positivity of the
Schubert coefficient.


\smallskip

\subsection{Geometric background and motivation}\label{ss:intro-geo-back}
The literature on Schubert calculus is much too large to be
reviewed here.  We refer to \cite{AF,Ful97,Knu16,Man01} for geometric
and combinatorial introductions, and to \cite{Knu22} for a recent overview.
Let us single out \cite{BS00} and \cite{BK06}, where Schubert vanishing
was studied in the context of representation theory and \emph{Horn's inequalities}
\ts describing possible spectra of
three Hermitian matrices satisfying $A+B=C$.

Now, recall the \emph{Grassmannian structure coefficients}, in which the
Schubert classes are pulled back from Grassmannians ${\sf G}/{\sf P}$,
where ${\sf P}$ is a maximal parabolic subgroup. As with the full flag varieties,
these Grassmannians have decompositions into Schubert cells \ts $[X_{{\lambda}}]$.
Taking the pullback by the projection \ts $\pi:{\sf G}/{\sf B}\twoheadrightarrow {\sf G}/{\sf P}$ \ts
embeds \ts $H^*({\sf G}/{\sf P})$ \ts as a subring of \ts $H^*({\sf G}/{\sf B})$.
Thus by specializing our algorithm to the appropriate Grassmannian elements,
we obtain a probabilistic poly-time algorithm to decide the vanishing of ordinary
and maximal isotropic Grassmannian structure constants.  We use notation \ts $c_{\la,\mu}^{\nu}(Y)$
\ts to denote these constants in type~$Y$ for \ts $k=3$.

In type A, we have \ts ${\sf G}/{\sf P}\simeq {\sf Gr}_{k,n}$ \ts is the
ordinary Grassmannian, the space of $k$-dimensional planes in ${\mathbb C}^n$.
Here, the Schubert structure constants \ts $c_{\la\mu}^{\nu}=c_{\la\mu}^{\nu}(A)$ \ts
are the \emph{Littlewood--Richardson $(${\rm LR}$)$ coefficients}, which are
extremely well studied in the literature, see e.g.\  \cite{Ful97,Sta99}.
Famously, LR coefficients are the structure constants of Schur polynomials
and have several combinatorial interpretations, see a long list in \cite[$\S$11.4]{Pak-OPAC}.

In a major breakthrough, Knutson and Tao \cite{KT99} established the
\emph{saturation conjecture} in type~$A$:
\begin{equation}\label{eq:saturation}
c^{\la}_{\mu\nu}(A) >0 \quad \Longleftrightarrow \quad c^{t\la}_{t\mu, t \nu}(A) >0 \ \ \text{\em for any \ $t\ge 1 \ts$}.
\end{equation}
In \cite{DeLM06,MNS12}, the authors independently observed that the
saturation property~\eqref{eq:saturation} implies that the vanishing
of LR coefficients can be solved by a linear program.  This gives:

\begin{thm}[{\rm \cite{KT99,DeLM06,MNS12}}{}] \label{t:LR}
The vanishing of LR-coefficients \ts $\big\{c^\la_{\mu\nu}(A) =^? 0\big\}$ \ts
can be decided in deterministic polynomial time.
\end{thm}

This result is exceptional and despite numerous conjectural generalizations
(see e.g.~\cite{Kir04}), it does not seem to extend much beyond this
narrow setting, see~$\S$\ref{ss:finrem-sat}.  Below we give a complexity
comparison of our probabilistic approach and the deterministic approach
as in the theorem.

The case of types B--D is also quite interesting and extensively studied.
We refer to \cite{BH95} for detailed overview of Schubert calculus in these
types, to \cite{PS08} for recursive Horn type formulas for the vanishing
problem, and to \cite{Purbhoo06} for detailed combinatorial
investigations of the vanishing. See also \cite{Sea16} for discussion
on the combinatorial formulas for the structure coefficients for other
choices of non-maximal isotropic Grassmannians in types B--D.

In this case, we highlight those \ts ${\sf G}/{\sf P}$ \ts
that are maximal isotropic Grassmannians
with respect to the appropriate skew--symmetric or symmetric bilinear form.
In type $C$, we have \ts $c_{\la,\mu}^\nu(C)$ \ts
as the structure constants of $Q$-Schur polynomials \ts $Q_{\la}$ \ts \cite{Pra91}.
Similarly, in types $B/D$, we have \ts $c_{\la,\mu}^\nu(B)$ \ts and \ts $c_{\la,\mu}^\nu(D)$ \ts
as the structure constants of $P$-Schur polynomials \ts $P_{\la}$,
ibid.  We refer to \cite{HH92} for the extensive treatment of $Q$- and $P$-Schur polynomials
in connection to projective representation theory of the symmetric group; see also
\cite[$\S$III.8]{Mac95} for the symmetric functions approach.

The type $B/C/D$ structure coefficients do not satisfy saturation (cf.~\cite[Rem.~7.7]{RYY22}), 
 so the arguments of \cite{DeLM06,MNS12} may not be mirrored directly.
In fact, our Theorem~\ref{t:main} gives the first poly-time
algorithm for the vanishing of LR coefficients in other types:

\begin{cor}\label{cor:main-LR}
For types \ts $Y \in \{A,B,C,D\}$, the problem
\begin{equation}\label{eq:LRV-def}
\LRV(Y) \ := \ \big\{c^\la_{\mu\nu}(Y) =^? 0\big\}
\end{equation}
can be decided in
probabilistic polynomial time when the input $\la,\mu,\nu$
is in unary.  More precisely, for $Y_n$ and every fixed $\delta>0$,
there is a probabilistic algorithm with one-sided error at most $\frac13$,
using $O(n^{4.5+\delta})$ arithmetic operations.
\end{cor}

The corollary addresses Problem~7.8 in \cite{RYY22} and
gets close to completely resolving it in the positive,
see below.  The unary input comes from the translation
of Schubert problems into Grassmannian notation;
in Theorem~\ref{t:LR} the usual (binary) input is used,
see~$\S$\ref{ss:finrem-alg} for further details.

\smallskip

\subsection{Complexity background and implications} \label{ss:intro-CS}
The algorithmic and complexity aspects of the Schubert vanishing
problem have also been heavily studied, both explicitly in the
computer algebra literature and implicitly in the algebraic
combinatorics literature.  In fact, the computational hardness
even for the well-studied \emph{2-step flag variety} setting
remains challenging, see e.g.\
\cite[Question~4.2]{ARY19}.  We refer to our extensive overview
in \cite[{\tt v2}, $\S$1.6]{PR-STOC}, to \cite[$\S$5]{StDY22}
for a combinatorial introduction to Schubert vanishing tests
in type~A, and to \cite[$\S$5.2]{BV08} for some motivating comments.

Below we give a brief discussion of prior complexity work on the
\emph{Schubert coefficient problem}
(the problem of computing Schubert coefficients)
and the Schubert vanishing problem.  We assume the
reader is familiar with standard complexity classes, which can be
found, e.g., in \cite{AB09,Gol08}.

The Schubert coefficient problem is known to be in \ts $\GapP=\SP-\SP$,
see \cite[Prop.~10.2]{Pak-OPAC} for type~$A$ (see also~\cite{PR-puzzle})
and \cite[Cor.~4.2]{PR-signed}
for other types.\footnote{In a combinatorial language,
this means that there are \emph{signed rules} for Schubert coefficients.}
Whether the Schubert coefficient problem is in~$\SP$ is a major open problem in the area,
see e.g.\ \cite[Problem~11]{Sta00},
\cite[Conj.~10.1]{Pak-OPAC} and \cite[O.P.~3.129]{BGP25}.\footnote{
In a combinatorial language, this problem asks for a \emph{combinatorial
interpretation} (rule) for Schubert coefficients.}  This problem has been
resolved in a number of special cases (see an overview in \cite{Knu22,PR-STOC}),
implying that Schubert vanishing is in $\coNP$ when restricted to each such case.

Notably, the classical \emph{Littlewood--Richardson rule} (see e.g.\ \cite[$\S$A1.3]{Sta99})
and the \emph{shifted LR rules} of Worley, Sagan and Stembridge (see e.g.\ \cite{CNO14}),
show that $\LRV(Y)\in \coNP$
\ts for \ts $Y \in \{A,B,C,D\}$.\footnote{Although usually stated for
the unary input, the result extends to the binary input, cf.\ \cite[$\S$5.2]{Pan23}.}
Additionally, in the language of \emph{root games}, Purbhoo
showed that Schubert vanishing is in $\NP$ in some special cases \cite{Purbhoo06}.

In \cite[Question~4.3]{ARY19}, the authors asked if \ts $\SV(A)$ \ts is \ts $\NP$-hard.
In the opposite direction, the authors conjectured that the problem is \ts $\coNP$-hard
\cite[Conj.~1.6]{PR-STOC}.  In a major advance \cite[Thm~1.1]{PR-Sigma}, we showed that
$$(\ast) \qquad
\SV(Y) \, \in \, \AM \cap \coAM \quad \text{for all} \quad Y \in \{A,B,C,D\},
$$
assuming the \emph{Generalized Riemann Hypothesis} ($\GRH$).  Prior to
\cite{PR-STOC,PR-Sigma}, it was believed that the problem is not in $\PH$.
In fact, prior to this work,
the best known unconditional upper bound was \ts $\SV(Y)\in \CEP \subseteq \PSPACE$,
even when restricting to type~$A$.
The inclusion is already nontrivial and follows from the $\GapP$ formulas
mentioned above.

In complexity theoretic language, our Main Theorem~\ref{t:main} proves (unconditionally) that
$$(\ast\ast) \qquad
\SV(Y) \, \in \, \coRP \quad \text{for all} \quad Y \in \{A,B,C,D\}.
$$
Here $\coRP$ consists of the decision problems for which a probabilistic
polynomial-time algorithm always returns the correct ``yes'' answer,
and returns the correct ``no'' answer with probability at least~$1/2$.
This is very low in the polynomial hierarchy, and we remind the reader of standard
inclusions:
$$
\poly \, \subseteq \, \coRP \, \subseteq \, \BPP \cap \coNP \,
\subseteq \,
\AM \cap \coAM \, \subseteq \, \Sigmap_2 \cap \Pip_2 \, \subseteq \, \PH \, \subseteq \, \PSPACE\ts.
$$
Corollary~\ref{cor:main-LR} similarly gives \ts $\LRV(Y)\in \coRP$ \ts
for the unary input, but the result is new only for \ts $Y\in \{B,C,D\}$ \ts
as Theorem~\ref{t:LR} gives \ts $\LRV(A)\in \poly$.

In conclusion, we mention that \ts $\poly=\RP=\coRP=\BPP$ \ts under standard derandomization
assumptions \cite{IW97}, see also a discussion in~\cite{PR-CR}.  In the opposite direction,
recall that Main Theorem~\ref{t:main} is proved via \emph{Polynomial Identity Testing}
($\PIT$), one of the main obstacles for derandomization \cite{SY09}.  It remains open
whether our approach can be used to show that Schubert vanishing is in~$\poly$.


\medskip

\section{Preliminaries} \label{s:prelim}

\subsection{Notation}\label{ss:prelim-not}
We use \ts $\nn=\{0,1,2,\ldots\}$ \ts and \ts $[n]=\{1,\ldots,n\}$.
Unless stated otherwise, the underlying field is always~$\cc$.
Let \ts $e_1,\ldots,e_n$ \ts denote the standard basis in~$\cc^n$.
We use bold letters \ts $\bx=(x_1,x_2,\ldots)$ \ts for collections of
variables, and \ts $\ov x$ \ts for their evaluations.
We write \ts $f\equiv g$ \ts for the equality of polynomials \ts $f,g\in \cc[\bx]$.

For a Weyl group $\cW$, we use \ts $\ell(w)$ \ts to denote the \emph{length}
of the element $w\in \cW$.  The \emph{long word} is an element $w_\circ\in \cW$ of
maximal length. We assume the reader
is familiar with standard notation of barred and unbarred elements of the
Weyl group \ts $\cW \simeq S_n \ltimes \zz_2^n$ \ts in types~$B$ and~$C$.
We view this Weyl group as the group of \emph{signed permutations} of~$[n]$,
see e.g.~\cite[$\S$14.1.1]{AF}. Here a signed permutation $w$ acts on $\mathbb{C}^n$ such that

\[
w(e_i) \, := \,
\begin{cases}
  e_{j} & \ \text{ if } \ w_i=j,\\\
  -e_{j} & \ \text{ if } \ w_i=\overline{j}.
\end{cases}
\]

Recall the following standard notation for {almost simple algebraic groups}.
We have the \emph{special linear group} ${\sf SL}_n(\cc)$, the \emph{odd special orthogonal group}
${\sf SO}_{2n+1}(\cc)$, the \emph{symplectic group} ${\sf Sp}_{2n}(\cc)$ and
the \emph{even special orthogonal group}  ${\sf SO}_{2n}(\cc)$.  These groups
correspond to \emph{root systems} $A_{n-1}$, $B_n$, $C_n$ and $D_n$,
and are called  \emph{groups of type}~$A$, $B$, $C$ and~$D$, respectively.

To distinguish the types, we use parentheses or subscripts in LR and
Schubert coefficients, e.g.\ \ts $c^\la_{\mu \nu}(A)$ \ts and \ts $c_{\<A\>}(u,v,w)$.
We omit the dependence on the type when it is clear from the context.

\smallskip

\subsection{Polynomial identity testing} \label{ss:proof-PIT}
In this paper we use the following textbook result:


\begin{lemma}[{\rm Schwartz--Zippel Lemma}{}]\label{lem:SZlem}
For a field ${\mathbb F}$,
    let $Q\in{\mathbb F}[x_1,x_2,\ldots,x_n]$
 be a non-zero polynomial with degree $d\geq 0$ over ${\mathbb F}$. Take  $S\subset {\mathbb F}$ be a finite set.
 Then:
 \[
 {\bf P}\big[Q(c_1,c_2,\dots,c_n)=0\big] \ \leq \ \frac{d}{|S|}\.,
 \]
 where the probability is over random, independent and uniform choices of \ts $c_1,c_2,\dots,c_n\in S$.
\end{lemma}


This lemma is frequently used to test whether a polynomial given by an
arithmetic circuit is identically zero.  We refer to \cite{SY09} for an
extensive overview of complexity applications and many references.

\smallskip

\subsection{Types B and C}\label{ss:prelim-BC}
We will need the following well-known result relating the Schubert vanishing in
two types:

\begin{prop}\label{prop:BC}
$\SV(B)$ \ts coincides with \ts $\SV(C)$.
\end{prop}

\begin{proof}
First we note that both $\SO_{2n+1}(\mathbb{C})$ and ${\sf Sp}_{2n}(\mathbb{C})$ share the hyperoctahedral group as their Weyl group ${\mathcal W}$. We interpret ${\mathcal W}$ as signed permutations of $[n]$.
Let $\ze(\pi)$ count the number of sign changes in the signed permutation \ts
$\pi\in \mathcal{W}$.  It follows from  \cite[Thm~3]{BH95}, that
    \begin{equation}\label{eq:Schub-BC}
    c_{\<B\>}({u_1,\dots,u_k}) \, = \, 2^{a} \. c_{\<C\>}({u_1,\dots,u_k})\ts,
    \end{equation}
where
$$
a \. := \. \ze(w_{\circ}u_k) \. - \. \ze(u_1) \. - \. \ldots \. - \. \ze(u_{k-1}).
$$
This implies the result.
\end{proof}

\smallskip

\subsection{Purbhoo's criterion}\label{ss:proof-purbhoo}
Take type \ts $Y\in\{A,B,C,D\}$ \ts and let \ts ${\sf G}= {\sf G}_Y$ \ts be the
semisimple complex Lie group of type $Y$.
Here ${\sf G}$ is a matrix group inside an ambient vector space~$V$.
Let ${\sf N}$ \ts be the subgroup of unipotent matrices, giving
$$
{\sf N}\subset {\sf B}\subset{\sf G} \subset V.
$$
Let ${\mathfrak n}$ denote the Lie algebra of ${\sf N}$. Again, we view
${\mathfrak n}$ as a subspace of~$V$.
Finally, for a Weyl group element $w\in \mathcal{W}$, let \ts $R_w:={\mathfrak n}\cap \operatorname{Ad}(\dot{w}^{-1}){\mathfrak n}_{-}$. Here ${\mathfrak n}_{-}$ is the Lie algebra of ${\sf N}_{-}$ \ts the opposite subgroup of unipotent matrices, and $\operatorname{Ad}$ denotes the adjoint action.
\smallskip


Generalizing the number of inversions condition,
the \emph{dimension condition} \ts says that
\begin{equation}\label{eq:dim-cond}
c(u_1,\ldots,u_k) \. = \. 0 \quad \text{ if } \quad \inv(u_1)+\dots+\inv(u_k) \. \ne \. \dim({\mathfrak n})\ts.
\end{equation}
Then by specializing \cite[Cor.~2.6]{Purbhoo06} appropriately, we obtain:
\begin{cor}[{\rm Purbhoo's criterion}]\label{lem:Pur}
  For generic \ts $\rho_1,\ldots,\rho_k\in{\sf N}\subset{\sf G}$ and $u_1,\dots,u_k\in \mathcal{W}$ such that $\inv(u_1)+\ldots+\inv(u_k)=\dim({\mathfrak n})$, we have:
    \[c(u_1,\ldots,u_k)\. > \. 0 \quad \Longleftrightarrow \quad
    \rho_1 R_{u_1}\rho_1^{-1}\. + \. \ldots \. + \. \rho_k R_{u_k}\rho_k^{-1}=\mathfrak{n}.
    \]
\end{cor}

Using Corollary~\ref{lem:Pur}, it suffices to determine the dimension of the vector space
$$
H \, := \, \rho_1 R_{u_1}\rho_1^{-1} \. + \. \ldots \. + \. \rho_k R_{u_k}\rho_k^{-1} \quad
\text{for generic} \quad \rho_i\ts.$$

\medskip

\section{General setup} \label{s:setup}

In the setting of Purbhoo's criterion, we describe how to construct
bases for \ts $R_{u_i}$ \ts in~$\S$\ref{subsec:posR}. Then,
in~$\S$\ref{subsec:uniP}, we describe how to construct matrices~$\rho_i\ts$.
In~$\S$\ref{constr:T}, we combine these constructions to
obtain bases for each summand \ts $\rho_i R_{u_i}\rho_i^{-1}$.
From these, we obtain vectors $\pi_j$ which generate~$H$.

\subsection{Root systems}\label{subsec:posR}
The Weyl group $\mathcal{W}$ is generated by reflections $r_\ga$, indexed by roots  $\ga$ in a \emph{root system} $\Phi$. This root system $\Phi$ may be
partitioned in terms of its positive and negative roots: \ts $\Phi=\Phi_+\sqcup \Phi_-\ts$.
In Table~\ref{Tab:posRoot}, we describe the positive roots $\Phi_+$ in terms of vectors $e_i\ts$. Here $e_i$ denotes the $i$-th
elementary basis vector in the appropriate $\mathbb{C}^m$.

\begin{table}[hbt]
\renewcommand{\arraystretch}{1.5}
\begin{tabular}{ |p{1.42cm}||p{7cm}|p{5.6cm}|  }
 \hline
 ${\sf G}$& $\Phi_+$ & $\numInt{{\sf G}}$\\
 \hline
 $\textsf{SL}_n$   & $\{e_i-e_j \, : \, 1\leq i<j\leq n\}$ & $\{(i,j) \, : \, 1\leq i<j\leq n\}$\\
 $\textsf{SO}_{2n+1}$ & $\{e_i\pm e_j \, : \, 1\leq i<j\leq n\}\cup \{e_i \, : \, i\in[n]\}$ & $\{(i,j) \, : \, 1\leq i<j\leq 2n+1-i\}$ \\
 $\textsf{SO}_{2n}$    &  $\{e_i\pm e_j \, : \, 1\leq i<j\leq n\}$ & $\{(i,j) \, : \, 1\leq i<j\leq 2n-i\}$ \\
 \hline
\end{tabular}
{
\vskip.3cm
\caption{Positive roots and corresponding matrix entries.}
\label{Tab:posRoot}}
\end{table}

Define the integer $\numGp{{\sf G}}$, where
\[\numGp{{\sf G}} \, = \,
\begin{cases}
    \. n & \text{ if } \ {\sf G}=\textsf{SL}_n(\cc),\\
    \. 2n+1 & \text{ if } \ {\sf G}=\textsf{SO}_{2n+1}(\cc),\\
    \. 2n & \text{ if } \ {\sf G}=\textsf{SO}_{2n}(\cc).
\end{cases}
\]

To each $\ga\in\Phi_+\ts$, we may associate an $m\times m$ matrix, where $m=\numGp{{\sf G}}$. We define a distinguished subset $\numInt{{\sf G}}\subset[m]\times[m]$ as outlined in Table~\ref{Tab:posRoot}.
We then construct a bijection \ts $\phi:\numInt{\sf G}\rightarrow \Phi_+$\ts, detailed below.

\smallskip

\nin
{\small (A)}  \ts For \ts $\textsf{SL}_n$ \. take \. $\phi(i,j):=e_i-e_j\ts$.

\smallskip

\nin
{\small (B)}  \ts  For \ts $\textsf{SO}_{2n+1}$ \ts take
 \[
\phi(i,j) \, := \,
\begin{cases}
  \ \ts  e_i-e_j & \ \text{ if } \ j\leq n,\\\
   \. e_i & \ \text{ if } \ j=n+1, \\
  \. e_i+e_{2n+2-j} & \ \text{ if } \ j>n+1.\\\

\end{cases}
\]
\smallskip

\nin
{\small (D)}  \ts  For \ts $\textsf{SO}_{2n}$ \ts take
    \[
\phi(i,j) \, := \,
\begin{cases}
  e_i-e_j & \ \text{ if } \ j\leq n,\\
    e_i+e_{2n+1-j} & \ \text{ if } \ j> n.
\end{cases}
\]

For \ts $\ga\in\Phi_+$\ts, set
\ts ${\rm E}_{\ga}'$ \ts to be the $m\times m$ matrix with a $1$
in position $\phi^{-1}(\ga)$ and $0$ elsewhere. For \ts $\textsf{SL}_n$ \ts
take \ts ${\rm E}_{\ga}:={\rm E}_{\ga}'\ts$.
For \ts $\textsf{SO}_{2n+1}$ \ts and \ts $\textsf{SO}_{2n}\ts$,
define
\ts ${\rm E}_{\ga}:={\rm E}_{\ga}'-{\rm D}_{m}({\rm E}_{\ga}')^{T}{\rm D}_{m}\ts$,
where  \ts ${\rm D}_{m}$ \ts is the antidiagonal matrix, \emph{i.e.}, the $m\times m$ matrix with ones on the main antidiagonal and zeroes elsewhere. Note that here we view
$\textsf{SO}_m=\{g\in \textsf{GL}_m \, : \, g^{T}{\rm D}_{m}g={\rm D}_{m}, \det(g)=1\}$.

\smallskip
\subsection{Generic unipotent subgroup elements}\label{subsec:uniP}

Let $m:=\numGp{{\sf G}}$.
We now describe how to construct an
upper unitriangular $m\times m$ matrix \ts $\rK$ \ts which lies
in \ts ${\sf N}\subset{\sf B}\subset{\sf G}$.
Define:
\[
\kappa_{ij}=
\begin{cases}
    \, \al_{ij} & \text{ if } \  i<j \ \text{ and } \ (i,j)\in \numInt{\sf G} \ts,\\
     \, z_{ij} & \text{ if } \  i<j \ , \ (i,j)\not\in \numInt{\sf G} \ts, \text{ and } i+j\neq m+1,\\
    \, 0 & \text{ otherwise}.
\end{cases}
\]
Here we treat \ts $\al_{ij}$ \ts as parameters and set \ts $z_{ij} = - \al_{(m+1-j)(m+1-i)}$. Let $\kappa=(\kappa_{ij})$.

For ${\sf G}=\textsf{SL}_n(\cc)$, set $\rK:={\rm I}_{m}+\kappa$.
For \ts ${\sf G}=\textsf{SO}_{2n+1}(\cc)$ \ts and \ts
${\sf G}=\textsf{SO}_{2n}(\cc)$, we use the Cayley transform to construct $\rK\in{\sf G}$ from $\kappa$. In particular, set $\rK=({\rm I}_{m}+\kappa)^{-1}({\rm I}_{m}-\kappa)$. By construction,
$\rK$ is upper unitriangular.
It is straightforward to confirm $\rK\in{\sf G}$ and that such elements are dense in ${\sf N}$. To check $\rK\in{\sf G}$, one need only confirm \ts $\rK^T \cdot {{\rm D}_{m}}\cdot \rK \, = \, {{\rm D}_{m}}\ts$,
and note that \ts $\det(\rK)=1$.

\smallskip
\subsection{Main construction}\label{constr:T}
Let $m=\numGp{{\sf G}}$.
With Corollary~\ref{lem:Pur} in mind, consider the vector space
\[
H \, = \, \rho_1 R_{u_1}\rho_1^{-1} \. + \. \dots \. + \.\rho_k R_{u_k}\rho_k^{-1}.
\]
Let \ts $d:=\dim{{\sf G}/{\sf B}}$  \ts and note \ts $\dim \mathfrak{n}=d$.
Using the dimension condition, we assume
\begin{equation}\label{eq:inv-sum}
\inv(u_1) \. + \. \ldots \. + \. \inv(u_k) \, = \, d\ts.
\end{equation}
Additionally, we assume \ts $\inv(u_i)\geq 1$ \ts for each $i\in[k]$. Combining these assumptions gives \ts $k\leq d$.

First recall the definition, \ts $R_w:={\mathfrak n}\cap \operatorname{Ad}(\dot{w}^{-1}){\mathfrak n}_{-}$ for \ts $w\in \mathcal{W}$.
Alternatively, $R_w$ is the subspace of ${\mathfrak n}$ generated by basis elements
${\rm E}_\ga$ (see ~$\S$\ref{subsec:posR}) for $\ga\in \Phi_+(w)$, where
\[
\Phi_+(w)\, := \, \big\{\be\in \Phi_+ \, : \, w\be\in \Phi_{-} \big\}.
\]
Now we construct bases for the spaces \ts $R_{u_i}$ for $i\in[k]$ \ts:
\begin{align*}
    S_{u_i} \, & := \, \big\{x_{{\ga},i} \. {\rm E}_{\ga} \, : \, \ga\in \Phi_+(u_i)\big\}.
\end{align*}

 Note the number of positive roots \ts $|\Phi_+|=d=O(m^2)$.
Let \ts $\bx:= \big\{x_{\ga,i}\big\}$ \ts be the set of those variables appearing in the above collection.
Since~\eqref{eq:dim-cond} holds,
we have the following:
\begin{equation}\label{eq:c1}
\sum_{i=1}^k\. \inv(u_i) \, = \, \sum_{i=1}^k \. \dim(R_{u_i}) \, = \, \sum_{i=1}^k \. \big|S_{u_i}\big| \. = \. d\ts.
\end{equation}

Then construct generic matrices \ts $\rho_1,\ldots,\rho_k \in{\sf N}$ \ts
as outlined in $\S$\ref{subsec:uniP}, using the formal parameters
\ts $\al_{j\ell}^{(i)}$.
Define \ts $\bal:=\big\{\al_{j\ell}^{(i)}\big\}$ \ts to be
the set of parameters, respectively,
appearing in some \ts $\rho_i$, where \ts $i\in [k]$. Then \ts $|\bal|\le k\cdot m^2 = O(m^4)$.

For each $i\in[k]$, construct bases for summands \ts $\rho_i R_{u_i}{\rho_i}^{-1}$\ts:
\begin{align*}
    T_{u_i} \, &:= \, \rho_i \ts S_{u_i} \ts {\rho_i}^{-1} \, = \, \big\{ \rho_i \cdot g \cdot {\rho_i}^{-1}\. : \. g \in S_{u_i}\big\}.
\end{align*}
Using~\eqref{eq:c1}, we find \ts $|T_{u_1}|+\ldots + |T_{u_k}|=d$.

Let $\tau$ be the map on $m\times m$ matrices defined by restricting to the matrix
entries in positions $\numInt{{\sf G}}$.
By their definition, matrices in ${\mathfrak n}$
are determined by their entries in positions \ts $\numInt{{\sf G}}$.
Thus, \ts $\dim(\operatorname{span}(T_{u_i}))=\dim(\operatorname{span}\big(\tau(T_{u_i})\big))$ \ts for each~$i\in[k]$.
Take
$$
T \, := \, \bigcup_{i\in[k]} \. \tau(T_{u_i})\ts,
$$
and let \ts $T=\{\pi_i \.:\. i\in [d]\}$.
Using the fact that \ts $|\numInt{{\sf G}}|=d$, we view each \ts $\pi_i\in T$ \ts as a $d$-vector. Note that $\tau$ gives a linear isomorphism from $\mathfrak{n}$ to $\mathbb{C}^d$, so $\dim(H)=\dim(\operatorname{span}(T))$.

Finally, consider the $d\times d$ matrix $M_\bx$ with column vectors $\pi_j\ts$. Let $M$ denote $M_\bx$ after setting $x_{\ga,i}=1$ for all $x_{\ga,i}\in \bx$. Note $x_{\ga,i}$ simply scales the $\ga,i$ column of $M_\bx$ by $x_{\ga,i}$. Thus $\det(M)\equiv 0$ if and only if  $\det(M_\bx)\equiv 0$.
Using Purbhoo's criterion, the Schubert vanishing problem reduces
to determining if the matrix $M$ is singular.  Since $M$ is a
matrix with entries polynomials in $\zz[\bal]$, we can apply the
Schwartz--Zippel Lemma~\ref{lem:SZlem}.  We do this carefully
in the next section.  We also note that the construction above
is similar to that in~\cite{PR-Sigma}.

\medskip

\section{Proof of the main theorem} \label{s:proof}

By Proposition~\ref{prop:BC}, it suffices to consider only types $A$, $B$ and~$D$.
For clarity, we consider types~$A$ and types~$B/D$ separately.

\smallskip

\subsection{Algorithm for $\textsf{SL}_n\ts$}\label{ss:proof-SL}
Considering the converse of Corollary~\ref{lem:Pur}, we let: 
\begin{equation}\label{eq:purRHS}
  c(u_1,\ldots,u_k)\. = \. 0 \quad \Longleftrightarrow \quad
    \rho_1 R_{u_1}\rho_1^{-1}\. + \. \ldots \. + \. \rho_k R_{u_k}\rho_k^{-1} \, \subsetneq \, {\mathfrak n}\ts.
\end{equation}

Construct $M$ as in~$\S$\ref{constr:T}.
Then the right-hand side of~\eqref{eq:purRHS} holds if and only if
$\det(M)\equiv 0$. Here, the condition $\det(M)\equiv0$ indicates that $\det(M)$ is identically zero.  We denote \ts $D(\bal):=\det(M)$.

In this case, $D(\bal)\in\mathbb{Z}[{\bal}]$ \ts since
\begin{equation}\label{eq:cofactor}
  (\rho_i^{-1})_{j_1,j_2}\, = \, \frac{1}{\det(\rho_i)} \, C^{(i)}(j_2,j_1),
\end{equation}
 where \ts $C^{(i)}(j_2,j_1)$ \ts is the cofactor of $\rho_i$ for $j_1, j_2\in[n]$.
 Thus by construction, we have \ts $\det(\rho_i)=1$, so \ts $(\rho_i^{-1})_{j_1,j_2}\in \mathbb{Z}[{\bal}]$.
Again, by~\eqref{eq:cofactor}, the expressions in $(\rho_i)^{-1}$ expanded in terms of $\al$ will be a polynomial of degree  $n-1$.
Thus matrix entries in $M$ are polynomials of degree at most  $n+1$.
Thus
$D(\bal)$ \ts has degree at most at most  $d(n+1)$.

Pick values for $\ov \al$ randomly over integers in $[p]$,
where $p\in\mathbb{Z}_{> 0}$. Let $\ov{\rho_i}$ denote $\rho_i$ evaluated at $\ov \al$.
Then compute $(\ov{\rho_i})^{-1}$.
By the Schwartz--Zippel Lemma~\ref{lem:SZlem}, if \ts $D(\bal)\not\equiv0$, we have:
\[{\bf P}\big[D(\ov \al)=0\big] \, \leq \, \frac{d(n+1)}{p}\..\]
Note that if \ts $D(\bal)\equiv0$, we have:
\[
{\bf P}\big[D(\ov \al)=0\big] \. = \. 1.
\]

We test $D(\ov \al)=0$ in polynomial time for these sampled values in $[p]$.
The probability of error is less than \ts $\frac{1}{3}$ \ts if we take
\ts $p>\frac{3}{2}n(n^2-1)$.
Thus by Corollary~\ref{lem:Pur}, the algorithm for deciding
\ts $\SV(A)$ \ts satisfies the bullet conditions in the Main
Theorem~\ref{t:main} with $\ve=\frac{1}{3}\ts$.

\smallskip

\subsection{Algorithm for $\textsf{SO}_m\ts$}\label{ss:proof-SO}
This case is similar to the case of $\textsf{SL}_n\ts$,
but we include the details for completeness.
Considering the converse of Corollary~\ref{lem:Pur}, we examine the equation
\begin{equation}\label{eq:purRHSBD}
  c(u_1,\ldots,u_k)\. = \. 0 \quad \Longleftrightarrow \quad
    \rho_1 R_{u_1}\rho_1^{-1}\. + \. \ldots \. + \. \rho_k R_{u_k}\rho_k^{-1} \, \subsetneq \, {\mathfrak n}\ts.
\end{equation}

Construct $M$ as in~$\S$\ref{constr:T}. Let $m=\numGp{{\sf G}}$.
Let $\kappa_i$ be as in~$\S$\ref{subsec:uniP} such that $\rho_i=({\rm I}_{m}+\kappa_i)^{-1}({\rm I}_{m}-\kappa_i)$.
Then the right-hand side of~\eqref{eq:purRHSBD} holds if and only if
$\det(M)\equiv 0$.
We denote $D(\bal):=\det(M)$.

Again $D(\bal)\in\mathbb{Z}[{\bal}]$ using~\eqref{eq:cofactor} since again $\det(\rho_i)=1$ for each $i\in[k]$.
By~\eqref{eq:cofactor}, the expressions in $\rho_i$ and $(\rho_i)^{-1}$ expanded in terms of $\al$ will be polynomials of degree at most  $m-1$.
Then matrix entries in $M$ are polynomials of degree at most  $2m+1$.
Thus
$D(\bal)$ \ts has degree at most $d(2m+1)$.

Pick values for $\ov \al$ randomly over integers in $[p]$, where $p\in\mathbb{Z}_{> 0}$. Then compute \ts $\ov{\rho_i}=({\rm I}_{m}+\ov{\kappa_i})^{-1}({\rm I}_{m}-\ov{\kappa_i})$ \ts
and \ts
 $(\ov{\rho_i})^{-1}$.
By the Schwartz--Zippel Lemma~\ref{lem:SZlem}, if \ts $D(\bal)\not\equiv0$, we have:
\[{\bf P}\big[D(\ov \al)=0\big]\, \leq \.  \frac{d(2m+1)}{p}\..
\]
Note that if \ts $D(\bal)\equiv0$, we have:
\[
{\bf P}\big[D(\ov \al)=0\big] \. = \. 1.
\]

We test $D(\ov \al)=0$ in polynomial time for these sampled values in $[p]$. The probability of error is less than $\frac{1}{3}$ if we take $p>3n^2(2m+1)\ts$.
Thus by Corollary~\ref{lem:Pur},  the algorithm for deciding both
\ts $\SV(B)$ \ts and \ts $\SV(D)$ \ts corresponding to odd and even~$m$, respectively,
satisfy the bullet conditions in the Main
Theorem~\ref{t:main} with $\ve=\frac{1}{3}\ts$.

\smallskip

\subsection{Algorithm outline}\label{ss:proof-algorithm}
For clarity, and to ease the complexity analysis below, we give a concise
outline of the algorithm in all types \. $A$, $B/D$. Note that types
$B$ and~$C$ are equivalent by Proposition~\ref{prop:BC}.

\medskip

{\bf Input:} \. $u_1,\ldots,u_k \in \cW$, where $u_i\neq\rm{id}$ for $i\in[k]$

{\bf Decide:} \. $\big[c(u_1,\dots,u_k)=^{?}0\big]$

\begin{itemize}
    \item  If $\sum_{i=1}^k\ell(u_i)\neq d$, output $c(u_1,\dots,u_k) \ = \ 0$.
   \item  Otherwise, let
    $$p \ := \ \left\{\aligned  \tfrac{3}{2}\. n(n^2-1)+1 \qquad & \text{if} \quad {\sf G}={\sf SL}_n \\
     \ 3\left\lfloor \tfrac{m}{2} \right\rfloor^2(2m+1)+1 \qquad & \text{if}  \quad {\sf G}={\sf SO}_m
     \endaligned\right.
    $$

    \item \underline{For all} \. $i\in[k]:$

    $\circ$  \  Generate strictly upper triangular matrices $\kappa_{i}$ with random entries $\alpha_{j\ell}\in [p]$, see~$\S$\ref{subsec:uniP}.

    $\circ$  \   Compute $\rK_i$ using $\kappa_{i}$ and set $\rho_{i}:=\rK_i\ts$, see~$\S$\ref{subsec:uniP}.

    $\circ$  \ Compute inverse matrices \ts $\rho_{i}^{-1}$.

    \item \underline{For all} \. $i\in[k]$ \ts and $\ga\in \Phi_+(u_i):$


    $\circ$  \  Compute matrices \ts $T_{{\ga},i}=\rho_{i}\big( \. {\rm E}_{\ga}\big)\rho_{i}^{-1}$.

    $\circ$  \  Record the entries of \ts $T_{{\ga},i}$ \ts in positions \ts $\numInt{{\sf G}}$ \ts as a vector \ts $v_{{\ga},i}\ts$.\footnote{Formally, we need to do this under a fixed order.
    Take, e.g., a lexicographic order on matrix positions.}
    \item Let \. $M$ \. be the matrix with column vectors \. $v_{{\ga},i}$\ts, over all \ts $\ga\in \Phi_+(u_i)$ \ts and  \ts $i\in[k]$.
\end{itemize}

{\bf Output:}
$$
\left\{\aligned \ c(u_1,\dots,u_k) \ = \ 0 \quad & \quad \text{if} \ \ \det(M)=0, \\
c(u_1,\dots,u_k) \ > \ 0 \quad & \quad \text{if} \ \ \det(M)\ne 0.
\endaligned
\right.
$$

\smallskip

\subsection{Algorithm analysis and proof of Theorem~\ref{t:main}}\label{ss:proof-analysis}
From the discussion in~$\S$\ref{ss:proof-SL} and~$\S$\ref{ss:proof-SO}, the Algorithm
above is always correct when it outputs \ts $[c(u_1,\dots,u_k)>0]$,
and its error probability in one trial is less than \ts $\frac13$.
Run \ts $s$ \ts independent trials, and output \ts $[c(u_1,\dots,u_k)>0]$
if any trial does so.  The error probability is then less than \ts $3^{-s}$.
For \ts $0<\ve\le\frac13$ \ts as in the theorem, take
$$
s \. := \. \left\lceil\log_3\frac1{\ve}\right\rceil
   \. = \. O\left(\log\frac1{\ve}\right),
$$
so that \ts $3^{-s}\le\ve$.

We first count arithmetic operations.  Write \ts $m:=\numGp{{\sf G}}$ \ts
for the side length of the group matrices.  For each \ts $i\in[k]$, let
\ts $\kappa_i$ \ts denote the strictly upper triangular parameter matrix used to
construct \ts $\rho_i$ \ts in~$\S$\ref{subsec:uniP}.
In all classical types considered here, \ts $m=\Theta(n)$ \ts and
\ts $d=\dim{\mathfrak n}=O(n^2)$.  Recall the standing assumption
\ts $k\le d$, with every \ts $u_i\ne{\rm id}$.
Computing the lengths and the sets \ts $\Phi_+(u_i)$ \ts takes
\ts $O(kn^2)$ \ts elementary steps.  If the initial length check fails,
the algorithm terminates.  Otherwise,
$$
\sum_{i=1}^k|\Phi_+(u_i)| \. = \. \sum_{i=1}^k\ell(u_i) \. = \. d,
$$
so the second loop performs exactly \ts $d$ \ts root-vector conjugations.

Let \ts $\om$ \ts denote the exponent of square matrix multiplication over
\ts $\mathbb Q$, that is, the infimum over all \ts $\tau$ \ts
for which two \ts $r\times r$ \ts matrices can be multiplied using
\ts $O(r^\tau)$ \ts arithmetic operations, see e.g.\ \cite[$\S$15.1]{BCS97}.
Note that the matrix multiplication exponent is invariant under extension
of the ground field; hence the bound over $\cc$ also applies over $\qqq$,
see \cite[$\S$15.3]{BCS97}.
In other words, for every fixed \ts $\eta>0$, we may choose a multiplication
algorithm using \ts $O(r^{\om+\eta})$ \ts arithmetic operations.
Recall that the same bound holds for matrix inversion and determinant
computation by the standard reductions, see \cite[$\S$16.4]{BCS97}.

Constructing and inverting the \ts $k$ \ts group matrices therefore costs
\ts $O(k \ts m^{\om+\eta})$ \ts arithmetic operations.
The \ts $d$ \ts conjugations cost \ts $O(d\ts m^{\om+\eta})$,
and the final \ts $d\times d$ \ts determinant costs
\ts $O(d^{\om+\eta})$.
Generating the parameters and assembling the columns of \ts $M$ \ts use
\ts $O(km^2+d^2)$ \ts additional scalar steps.
Together with preprocessing, these additional costs are dominated by the
three displayed matrix-operation bounds, since \ts $\om\ge2$.
Thus one trial costs at most
$$
\begin{aligned}
O\!\left(km^{\om+\eta}+dm^{\om+\eta}+d^{\om+\eta}\right)
&\ =\ O\!\left(n^{\om+2+\eta}+n^{2\om+2\eta}\right) \ =\ O\!\left(n^{2\om+2\eta}\right).
\end{aligned}
$$
Here we used \ts $m=O(n)$, the inequality \ts $k\le d=O(n^2)$,
and \ts $\om\ge 2$.
For the fixed \ts $\delta>0$ \ts as in the theorem, choose
\ts $\eta=\delta/2$.  The factor \ts $2$ \ts multiplying \ts $\eta$ \ts
comes from substituting \ts $d=O(n^2)$ \ts in \ts $d^{\om+\eta}$.
This gives \ts $O(n^{2\om+\delta})$ \ts arithmetic operations per trial.
Thus \ts $\delta$ \ts is an arbitrarily small fixed allowance in the final
running-time exponent, independent of \ts $n$, \ts $k$, and the error
parameter \ts $\ve$.

We next establish polynomial bit complexity separately.
For this purpose, implement the matrix construction by finite nilpotent
expansions and the determinant computation by fraction-free elimination.
To sample using a bounded number of random bits, enlarge \ts $[p]$ \ts to
\ts $[q]$, where \ts $q=2^{\lceil\log_2p\rceil}$.
Then \ts $p\le q<2p=O(n^3)$, so the preceding Schwartz--Zippel estimates
remain valid, and each parameter uses exactly \ts $\log_2q=O(\log n)$ \ts
independent random bits.

Since \ts $\kappa_i^m=0$, we have
$$
\big({\rm I}_m\pm\kappa_i\big)^{-1}
   \. = \. \sum_{j=0}^{m-1}\big(\mp\kappa_i\big)^j.
$$
For \ts $j\ge1$, each entry of \ts $\kappa_i^j$ \ts is a sum of at most
\ts $m^{j-1}$ \ts products, each of magnitude at most \ts $q^j$.
It follows that the entries of \ts $\rho_i$, \ts $\rho_i^{-1}$, and their
intermediate construction matrices have bit length
\ts $O(m\log(mq))$.
Since the entries of \ts $ {\rm E}_\ga$ \ts are \ts $0,\pm1$, the same
bound holds for the conjugations and hence for the entries of \ts $M$:
$$
L \. = \. O\big(m\log(mq)\big) \. = \. O(n\log n).
$$
By Hadamard's inequality \cite[$\S$2.11]{BB61}, every minor of \ts $M$ \ts
has bit length at most
$$
O\big(d(L+\log d)\big) \. = \. O(n^3\log n).
$$
Fraction-free elimination with pivoting therefore computes \ts $\det(M)$ \ts
using polynomially many integer operations, with intermediate integers
of polynomial bit length: its entries are minors, and the products formed
before exact division also have polynomial bit length.
The finite expansions and matrix products above likewise require only
polynomially many operations on integers of polynomial bit length.
This proves polynomial bit time for each trial and for the amplified
algorithm.  It does not identify the bit-operation count with the
arithmetic-operation count.

Finally, multiplying the arithmetic cost per trial by \ts $s$, we obtain
$$
O\!\left(n^{2\om+\delta}\log\frac1{\ve}\right)
$$
arithmetic operations.
Substituting the OpenAI bound \ts $\om\le\frac94$ \ts
\cite{OAI26MM}, we obtain the bound as in the theorem.  This completes the
proof of Theorem~\ref{t:main}.
\qed

\medskip


\section{Final remarks}\label{s:finrem}

\subsection{} \label{ss:finrem-hist}
Schubert vanishing is a major problem with connections across areas, such as
representation theory, category theory, matroid theory and pole placement problem in linear
systems theory (see \cite{PR-STOC} for many references).  Let us quote Knutson's ICM paper:
\emph{``For applications $($including real-world engineering applications$)$ it is more
important to know that {\rm $[$Schubert$]$} structure constant is positive,
than it is to know its actual value''}~\cite[$\S$1.4]{Knu22}.

More broadly, the vanishing of structure constants in Algebraic Combinatorics
plays a central role in \emph{Geometric Complexity Theory} (GCT), as discussed
at length in \cite{Mul09,MNS12}.  Notably, an important part of GCT
was motivated by an observation that the saturation theorem implies
that vanishing of LR coefficients is in~$\P$, ibid.   We refer to
\cite[$\S$6.6.3]{Aar16} for a high level overview of this connection.
Let us mention that in \cite[$\S$3.7]{Mul09}, Mulmuley singled
out Schubert coefficients as
``one of the fundamental structural constants in representation
theory and algebraic geometry,'' whose vanishing needs to be understood.

\subsection{}\label{ss:finrem-alg}
As we mentioned in~$\S$\ref{ss:intro-CS}, in type $A$ the vanishing of LR
coefficients~\eqref{eq:LRV-def} can be decided by linear programming
via the saturation property.  An alternative approach 
was given by B\"urgisser and Ikenmeyer in \cite{BI13}.  Their algorithm
uses flows in hive graphs and is specifically designed
for the LR vanishing.  While we make no effort
to optimize our algorithm, below we include a brief comparison
of the time complexity of these algorithms, and the algorithm
we obtain in Corollary~\ref{cor:main-LR}.

We assume that the three partitions $\la,\mu,\nu$ are given in binary
as vectors in \ts $\nn^\ell$, so \ts $|\la|=|\mu|+|\nu|$ \ts and \ts $\ell(\la), \ell(\mu), \ell(\nu)\le \ell$.
Denote  by \ts $a:=\log_2 \la_1$ \ts the bit-size of the maximal part of the input.
According to \cite[p.~1640]{BI13}, the \emph{ellipsoid method}
applied to the hive polytopes given in \cite{KT99}, takes \ts $O(\ell^{10} a \ts \vk)$,
where $\vk$ denotes the cost of arithmetic operations.
%
By contrast, the B\"urgisser--Ikenmeyer (BI) algorithm
takes \ts $O(\ell^3 a\ts \vk)$ \. \cite[Thm~5.4]{BI13}.

Note that our probabilistic algorithm uses unary input, which makes the
complexity not directly comparable.  This is because for the standard
embedding of the LR vanishing into the Schubert vanishing, we have \ts
$n=(\ell+\la_1)$.  For comparison sake, assume that \ts $\la_1=\Theta(\ell)$.
In this case, the LP methods take \ts $O(\ell^{10+o(1)}\vk)$,  
while the BI~algorithm takes \ts $O(\ell^{3+o(1)}\vk)$
\ts in this case.

Now, the algorithm chosen in
Corollary~\ref{cor:main-LR} uses \ts $O(\ell^{4.5+\delta})$ \ts
arithmetic operations at constant error probability.
Thus, our arithmetic bound
has a smaller exponent than the LP bounds above, but a larger exponent
than the BI bound.  This comparison concerns arithmetic-operation counts;
the polynomial bit-time argument in~$\S$\ref{ss:proof-analysis} does not
give the same exponent for bit operations.  Note, however, that our analysis
in~$\S$\ref{ss:proof-analysis} is not especially sharp since the matrices
we are multiplying/inverting are very sparse (see Appendix~\ref{s:App}).  It would be interesting to
improve our analysis, especially in the Grassmannian case.

\subsection{} \label{ss:finrem-sat}
It is natural to ask if the saturation property~\eqref{eq:saturation} can be extended
beyond LR coefficients.  The {exuberance} which followed Knutson--Tao's proof
led to a plethora of potential generalizations, see e.g.\ a large compendium in~\cite{Kir04}.
Following the logic of \cite{DeLM06,MNS12}, such results could potentially give
\emph{deterministic} poly-time algorithms for the vanishing problems.
With few notable exceptions, almost none of these potential generalizations are proved,
and many have been refuted.  We refer to  \cite{PR-sat} for a disproof of Kirillov's
conjectural saturation property for Schubert coefficients, and further references.

\vskip.7cm

\subsection*{Acknowledgements}
We are grateful to Sara Billey, Allen Knutson, Greta Panova,
Kevin Purbhoo, Frank Sottile, David Speyer, Avi Wigderson, Weihong Xu and Alex Yong
for interesting discussions and helpful comments.
The first author was partially supported by the NSF grant CCF-2302173.
Unfortunately, this grant was suspended at the time the paper was written.
The second author was partially supported by the NSF MSPRF grant DMS-2302279.

\newpage


 \newpage

{\small
\appendix
\section{Examples} \label{s:App}

We illustrate the construction in~\S\ref{constr:T} for two problems in \ts ${\SL}_4(\cc)$ \ts
and for one problem in \ts $\SO_7(\cc)$.  In the first of the two \ts ${\SL}_4(\cc)$ \ts examples,
the outcome determines the coefficient vanishes, and in the second, the outcome determines
the coefficient is positive.

\smallskip

\subsection{Vanishing \ts ${\SL}_4$ \ts example}
\label{ex:A1zero}
Take $u=3214$, $v=1423$, and $w=4312$. This gives $w_{\circ}w=1243$.
We have:
\begin{align*}
    \Phi_+(u)&\, = \, \{e_1-e_2, e_1-e_3, e_2-e_3\},\\
    \Phi_+(v)&\, = \, \{e_2-e_3, e_2-e_4\},\\
    \Phi_+(w_{\circ}w)&\, = \, \{e_3-e_4\}.
\end{align*}
Using these roots, we construct the bases $S_u$, $S_v$, and $S_{w_{\circ}w}:$
\begin{align*}
   x_{1} \. {\rm E}_{e_1-e_2}=&\begin{pmatrix}
0 & x_1 & 0 & 0\\
0 & 0 & 0 & 0\\
0 & 0 & 0 & 0\\
0 & 0 & 0 & 0
\end{pmatrix}&
 x_{2} \. {\rm E}_{e_1-e_3}=&\begin{pmatrix}
0 & 0 & x_2 & 0\\
0 & 0 & 0 & 0\\
0 & 0 & 0 & 0\\
0 & 0 & 0 & 0
\end{pmatrix}&
   x_{3} \. {\rm E}_{e_2-e_3}=&\begin{pmatrix}
0 & 0 & 0 & 0\\
0 & 0 & x_3 & 0\\
0 & 0 & 0 & 0\\
0 & 0 & 0 & 0
\end{pmatrix}
\end{align*}
\begin{align*}
x_{4} \. {\rm E}_{e_2-e_3}=&\begin{pmatrix}
0 & 0 & 0 & 0\\
0 & 0 & x_4 & 0\\
0 & 0 & 0 & 0\\
0 & 0 & 0 & 0
\end{pmatrix}&
 x_{5} \. {\rm E}_{e_2-e_4}=&\begin{pmatrix}
0 & 0 & 0 & 0\\
0 & 0 & 0 & x_5\\
0 & 0 & 0 & 0\\
0 & 0 & 0 & 0
\end{pmatrix}&
   x_{6} \. {\rm E}_{e_3-e_4}=&\begin{pmatrix}
0 & 0 & 0 & 0\\
0 & 0 & 0 & 0\\
0 & 0 & 0 & x_6\\
0 & 0 & 0 & 0
\end{pmatrix}
\end{align*}
We note that in the algorithm, we set all $x_i=1$, so we will apply this evaluation in all proceeding computations.
We then build matrices $\rho_i$ as follows:
\[\rho_1=\begin{pmatrix}
1 & a_0 & a_1 & a_2\\
0 & 1 & a_3 & a_4\\
0 & 0 & 1 & a_5\\
0 & 0 & 0 & 1
\end{pmatrix}\qquad \rho_2=\begin{pmatrix}
1 & b_0 & b_1 & b_2\\
0 & 1 & b_3 & b_4\\
0 & 0 & 1 & b_5\\
0 & 0 & 0 & 1
\end{pmatrix}\qquad \rho_3=\begin{pmatrix}
1 & c_0 & c_1 & c_2\\
0 & 1 & c_3 & c_4\\
0 & 0 & 1 & c_5\\
0 & 0 & 0 & 1
\end{pmatrix}\]

After computing $\rho_1 S_{u}{\rho_1}^{-1}$, $\rho_2S_v{\rho_2}^{-1}$, and $\rho_3S_{w_{\circ}w}{\rho_3}^{-1}$, we restrict to the strictly upper triangular entries to build column vectors $\pi_i$. We illustrate this for the first basis element in $\rho_1 S_{u}{\rho_1}^{-1}$:
\begin{align*}
   & \rho_1 \.( {\rm E}_{e_1-e_2}){\rho_1}^{-1} \,   \\ & =\begin{pmatrix}
1 & a_0 & a_1 & a_2\\
0 & 1 & a_3 & a_4\\
0 & 0 & 1 & a_5\\
0 & 0 & 0 & 1
\end{pmatrix} \begin{pmatrix}
0 & 1 & 0 & 0\\
0 & 0 & 0 & 0\\
0 & 0 & 0 & 0\\
0 & 0 & 0 & 0
\end{pmatrix} \begin{pmatrix}
1 & -a_0 & a_0a_3-a_1 & -a_2+a_1a_5+a_0(a_4-a_3a_5)\\
0 & 1 & -a_3 & a_3a_5-a_4\\
0 & 0 & 1 & -a_5\\
0 & 0 & 0 & 1
\end{pmatrix} \\
& =\begin{pmatrix}
0 & 1 & -a_3 & a_3a_5-a_4\\
0 & 0 & 0 & 0\\
0 & 0 & 0 & 0\\
0 & 0 & 0 & 0
\end{pmatrix}
\xrightarrow{\tau}
[1 ,\. -a_3 ,\. (a_3a_5-a_4),\.0,\.0,\.0]^T
\end{align*}
Repeating this process for each basis element, we obtain $\{\pi_i\}$ to build the following matrix:
\[M=\begin{pmatrix}
1 & 0 & 0 & 0 & 0 & 0 \\
-a_3 & 1 & a_0 & b_0 & 0 & 0 \\
a_3a_5-a_4 & -a_5 & -a_0a_5 & -b_0b_5 & b_0 & c_1\\
0 & 0 & 1 & 1 & 0 & 0 \\
0 & 0 & -a_5 & -b_5 & 1 & c_3 \\
0 & 0 & 0 & 0 & 0 & 1
\end{pmatrix}. \]
One can check that \ts $\det(M)\equiv 0$ \ts in this case, so \ts $c_{u,v}^w=0$.

Rather than compute \ts $\det(M)\equiv 0$ \ts directly, in the algorithm
we instead randomly
evaluate the variables above in the interval $[91]$
to produce an evaluated matrix $\oa{M}$.
Of course, we will always have \ts $\det(\oa{M})=0$ \ts in this case.

\smallskip
 \subsection{Nonvanishing \ts ${\SL}_4$ \ts example}
\label{ex:A2pos}
Now we take $v=1342$ to consider the triple
$u=3214$, $v=1342$, and $w=4312$. Again $w_{\circ}w=1243$.
This time we have $\Phi_+(v)=\{e_2-e_4, e_3-e_4\}$.
Take $\rho_i$ as in Example~\ref{ex:A1zero}.

By the same process as Example~\ref{ex:A1zero}, we obtain $\{\pi_i\}$ to build the following matrix:
\[M=\begin{pmatrix}
1 & 0 & 0 & 0 & 0 & 0 \\
-a_3 & 1 & a_0x_3 & 0 & 0 & 0 \\
a_3a_5-a_4 & -a_5 & -a_0a_5 & b_0 & b_1 & c_1\\
0 & 0 & 1 & 0 & 0 & 0 \\
0 & 0 & -a_5 & 1 & b_3 & c_3 \\
0 & 0 & 0 & 0 & 1 & 1
\end{pmatrix}. \]

We find \ts $\det(M)\not\equiv 0$, so \ts $c_{u,v}^w>0$.

Now, the algorithm tests \ts $\{\det(M)\equiv^? 0\}$.
We randomly evaluate the variables above in the interval $[91]$ to produce an evaluated matrix $\oa{M}$. With probability $>2/3$, our random choices will result in $\oa{M}$ which will output $\det(\oa{M})\neq0$. Since the algorithm has one--sided error,
if \ts {\em \textbf{any}} \ts evaluation $\oa{M}$ produces $\det(\oa{M})\neq0$, then we have \ts $c_{u,v}^w>0$ \ts
with certainty.

For example, take
\[\oa{\rho_1} \, \gets \, \begin{pmatrix}
1 & 1 & 2 & 3\\
0 & 1 & 4 & 5 \\
0 & 0 & 1 & 6 \\
0 & 0 & 0 & 1
\end{pmatrix} \ ,\qquad \oa{\rho_2}\, \gets \, \begin{pmatrix}
1 & 7 & 8 & 9\\
0 & 1 & 10 & 11\\
0 & 0 & 1 & 12\\
0 & 0 & 0 & 1
\end{pmatrix} \ , \qquad\oa{\rho_3}\, \gets \, \begin{pmatrix}
1 & 13 & 14 & 15\\
0 & 1 & 16 & 17\\
0 & 0 & 1 & 18\\
0 & 0 & 0 & 1
\end{pmatrix}\..\]
This substitution gives \ts $\oa{M}$ \ts with \ts $\det(\oa{M})=36\neq0$,
so we conclude that \ts $c_{u,v}^w>0$.

\smallskip

\subsection{An \ts ${\sf SO}_7$ \ts example}\label{ex:Bpos}
Here we illustrate the construction in~\S\ref{constr:T} for a problem in \ts $\SO_{7}\ts$.
After the setup below, the remainder of the algorithm follows precisely like the \ts ${\SL}_4(\cc)$ \ts cases above.

In this example, we consider an example in the maximal isotropic Grassmannian of type~$B_3$.
Take $u=13\ol{2}$, $v=1\ol{3}\ol{2}$, and $w=3\ol{2}\ol{1}$. This gives
$w_{\circ}w=\ol{3}21$ since $w_{\circ}=\ol{1}\ol{2}\ol{3}$.
We have:
\begin{align*}
    \Phi_+(u)&\, = \, \{e_3,e_2+e_3\},\\
    \Phi_+(v)&\, = \, \{e_2, e_3, e_2+e_3\},\\
    \Phi_+(w_{\circ}w)&\, = \, \{e_1, e_1-e_2,e_1-e_3,e_2-e_3\}.
\end{align*}
Using these roots, we construct the bases $S_u$, $S_v$, and $S_{w_{\circ}w}$, where we again evaluate all $x_i=0$:
\begin{align*}
 {\rm E}_{e_3}\, = \, &\begin{pmatrix}
0 & 0 & 0 & 0& 0 & 0 & 0 \\
0 & 0 & 0 & 0 & 0 & 0 & 0 \\
0 & 0 & 0 & 1 & 0 & 0 & 0 \\
0 & 0 & 0 & 0 & -1 & 0 & 0 \\
0 & 0 & 0 & 0 & 0 & 0 & 0 \\
0 & 0 & 0 & 0 & 0 & 0 & 0 \\
0 & 0 & 0 & 0 & 0 & 0 & 0
\end{pmatrix}&
{\rm E}_{e_2+e_3}\, = \, &\begin{pmatrix}
0 & 0 & 0 & 0& 0 & 0 & 0 \\
0 & 0 & 0 & 0 & 1 & 0 & 0 \\
0 & 0 & 0 & 0 & 0 & -1 & 0 \\
0 & 0 & 0 & 0 & 0 & 0 & 0 \\
0 & 0 & 0 & 0 & 0 & 0 & 0 \\
0 & 0 & 0 & 0 & 0 & 0 & 0 \\
0 & 0 & 0 & 0 & 0 & 0 & 0
\end{pmatrix}\end{align*}
\begin{align*}
  {\rm E}_{e_2}\, = \, &\begin{pmatrix}
0 & 0 & 0 & 0& 0 & 0 & 0 \\
0 & 0 & 0 & 1 & 0 & 0 & 0 \\
0 & 0 & 0 & 0 & 0 & 0 & 0 \\
0 & 0 & 0 & 0 & 0 & -1 & 0 \\
0 & 0 & 0 & 0 & 0 & 0 & 0 \\
0 & 0 & 0 & 0 & 0 & 0 & 0 \\
0 & 0 & 0 & 0 & 0 & 0 & 0
\end{pmatrix}&
 {\rm E}_{e_1} \, = \, &\begin{pmatrix}
0 & 0 & 0 & 1 & 0 & 0 & 0 \\
0 & 0 & 0 & 0 & 0 & 0 & 0 \\
0 & 0 & 0 & 0 & 0 & 0 & 0 \\
0 & 0 & 0 & 0 & 0 & 0 & -1 \\
0 & 0 & 0 & 0 & 0 & 0 & 0 \\
0 & 0 & 0 & 0 & 0 & 0 & 0 \\
0 & 0 & 0 & 0 & 0 & 0 & 0
\end{pmatrix}\end{align*}
\begin{align*}
{\rm E}_{e_1-e_2}\, = \, &\begin{pmatrix}
0 & 1 & 0 & 0& 0 & 0 & 0 \\
0 & 0 & 0 & 0 & 0 & 0 & 0 \\
0 & 0 & 0 & 0 & 0 & 0 & 0 \\
0 & 0 & 0 & 0 & 0 & 0 & 0 \\
0 & 0 & 0 & 0 & 0 & 0 & 0 \\
0 & 0 & 0 & 0 & 0 & 0 & -1 \\
0 & 0 & 0 & 0 & 0 & 0 & 0
\end{pmatrix}&
 {\rm E}_{e_1-e_3}\, = \, &\begin{pmatrix}
0 & 0 & 1 & 0& 0 & 0 & 0 \\
0 & 0 & 0 & 0 & 0 & 0 & 0 \\
0 & 0 & 0 & 0 & 0 & 0 & 0 \\
0 & 0 & 0 & 0 & 0 & 0 & 0 \\
0 & 0 & 0 & 0 & 0 & 0 & -1 \\
0 & 0 & 0 & 0 & 0 & 0 & 0 \\
0 & 0 & 0 & 0 & 0 & 0 & 0
\end{pmatrix}\end{align*}
\begin{align*}
  {\rm E}_{e_2-e_3}\, = \, &\begin{pmatrix}
0 & 0 & 0 & 0& 0 & 0 & 0 \\
0 & 0 & 1 & 0 & 0 & 0 & 0 \\
0 & 0 & 0 & 0 & 0 & 0 & 0 \\
0 & 0 & 0 & 0 & 0 & 0 & 0 \\
0 & 0 & 0 & 0 & 0 & -1 & 0 \\
0 & 0 & 0 & 0 & 0 & 0 & 0 \\
0 & 0 & 0 & 0 & 0 & 0 & 0
\end{pmatrix}
\end{align*}
We then build matrices $\rho_i$ as follows:
\[\kappa_1 \, = \, \begin{pmatrix}
0 & \textcolor{blue}{a_0} & \textcolor{blue}{a_1} & \textcolor{blue}{a_2}& \textcolor{blue}{a_3} & \textcolor{blue}{a_4} & 0 \\
0 & 0 & \textcolor{blue}{a_5} & \textcolor{blue}{a_6} & \textcolor{blue}{a_7} & 0 & -a_4 \\
0 & 0 & 0 & \textcolor{blue}{a_8} & 0 & -a_7 & -a_3 \\
0 & 0 & 0 & 0 & -a_8 & -a_6 & -a_2 \\
0 & 0 & 0 & 0 & 0 & -a_5 & -a_1 \\
0 & 0 & 0 & 0 & 0 & 0 & -a_0 \\
0 & 0 & 0 & 0 & 0 & 0 & 0 \\
\end{pmatrix} \]
Here we have highlighted the matrix entries above with positions in $\numInt{{\sf SO}_{7}}$ in blue.
We now set \. $\rho_1:=({\rm I}_{7}+\kappa_1)^{-1}({\rm I}_{7}-\kappa_1)$.  We then
construct matrices  $\rho_2$ and $\rho_3$ similarly.
The remainder of the algorithm proceeds analogously as in Examples~\ref{ex:A1zero} and~\ref{ex:A2pos}.
As matrices become rather large, we omit the details.

}

\end{document}